\documentclass[11pt]{amsart}
\usepackage{amsthm}
\usepackage{graphics}
\usepackage{graphicx,psfrag}
\newtheorem {thm}{{Theorem}}[section]
\newtheorem{lem}[thm]{Lemma}

\newtheorem{defn}[thm]{Definition}
\newtheorem{cor}[thm]{Corollary}


\begin{document}
\title[$\mathbb Z^d$ Rectangle Multi-tower Theorem]{The $\mathbb Z^d$ Alpern multi-tower theorem for
  rectangles:  a tiling approach}
\author{Ay\c se A. \c Sah\.in}
\address{Department of Mathematical Sciences\\ DePaul University\\ 2320
North Kenmore Ave.\\ Chicago, IL 60614}
\email{asahin@condor.depaul.edu}
\thanks{This research was supported in part by a DePaul University Research Council Paid Leave.}
\maketitle

\begin{abstract}We provide a proof of the Alpern multi-tower theorem for $\mathbb Z^d$ actions.  We reformulate the theorem as a problem of measurably tiling orbits of a $\mathbb Z^d$ action by a collection of rectangles whose corresponding sides have no non-trivial common divisors.  We associate to such a collection of rectangles a special family of generalized domino tilings.  We then identify an intrinsic dynamic property of these tilings, viewed as symbolic dynamical systems, which allows for a multi-tower decomposition.
\end{abstract}

\section{Introduction}
The Rohlin Lemma is one of the fundamental results of ergodic
theory. In \cite{AlpRL} Alpern proved a generalization of the Rohlin
Lemma; he showed that given a free, measure preserving
transformation $T$ of a Lebesgue space $X$, one can decompose $X$
into $k$ towers provided that $k\ge 2$ and the heights of the towers
do not have a non-trivial common divisor. He also showed that
each tower can be made to take up any proportion of the space.  A
later, and shorter, proof was given in \cite{EP}. Alpern's
Multi-tower Theorem has also played an important role in the proof
of many recent important results in ergodic theory (see \cite{AlpRL}
for the original application; see \cite{Korn} and \cite{EP} for a
survey of more recent applications) and has been the subject of some
recent research activity (see for example \cite{BDM} and
\cite{EHP}).

The Rohlin Lemma for $\mathbb Z^d$ actions for towers with
rectangular shape was first shown by Conze \cite{C} and Katznelson
and Weiss \cite{KW}.  Their results were subsumed by the later work
of Ornstein and Weiss \cite{OW} who have generalized the Rohlin
Lemma to free actions of countable amenable groups.  It follows from
their work that for finite subsets $R$ of $\mathbb Z^d$ there is a
Rohlin Lemma with towers of shape $R$ if and only if $R$ tiles
$\mathbb Z^d$.

The Alpern Multi-tower Theorem was generalized to $\mathbb Z^d$ for rectangular
towers whose corresponding dimensions do not have a non-trivial
common divisor by Prikhodko \cite{Prikh} in 1999.  In this paper we give a different proof of the same result obtained independently by the author in May, 2005.  The author gave multiple presentations of the argument in this paper before learning from V. S. Prasad about Prikhodko's article in October, 2005.

The approach given here can be viewed as playing the role
of the \lq\lq simple" proof of the higher dimensional result
analogous to \cite{EP}.  In particular, establishing the desired
distribution of the towers is an easy part of our argument, whereas
it constitutes the bulk of the work in \cite{Prikh}.

The more significant difference in our approach lies in that we
identify an intrinsic and dynamic property of a family of tilings
which allows for a multi-tower decomposition with those tile shapes.
This formulation provides a fruitful direction for investigating a
more general multi-tower theorem.

More specifically, the $\mathbb Z^d$ Alpern Multi-tower Theorem for rectangles
can be reformulated as a problem of tiling orbits of an action.  It
states that for a.e. $x\in X$, the orbit of $x$ under the action $T$
can be measurably tiled by $k\ge 2$ rectangles whose corresponding
dimensions have do not have a common divisor. Further, these tilings
can be constructed so that each rectangular tile has a previously prescribed probability distribution.  Tilings of $\mathbb Z^d$ by rectangles can be viewed
as tilings of $\mathbb R^d$ by dominos where each domino has integer
dimensions, and whose vertices lie on the integer lattice.  In the
proof we present here we associate to each collection of required
tower shapes a domino tiling of $\mathbb R^d$ and we translate the
problem of decomposing the space into multiple towers into the
problem of finding a factor of the given system which is an
invariant measure on the associated domino tiling system.  This
point of view allows us to identify a fairly innocuous mixing property of
the tiling system (having a {\it uniform filling set}) as the key
ingredient necessary to construct such a factor.

There are many other well studied and interesting tilings which have
this property including lozenge tilings and square ice, (see
\cite{RS4}) and are a natural place to start investigating the
possibility of a general multi tower theorem. Namely, given a collection
of shapes which tile $\mathbb Z^d$ can one establish necessary and
sufficient conditions for there to be an Alpern Theorem with towers of
these shapes?

We establish some notation to state the multi-tower theorem more formally.  For
$\vec w=(w_1,\cdots,w_d)\in\mathbb N^d$ we set $R_{\vec
w}=\prod_{i=1}^d [0,w_i-1]\cap\mathbb Z$.   We call $R_{\vec w}$ a
rectangle in $\mathbb Z^d$.  Given a $\mathbb Z^d$ action $T$ on a
space $X$, a subset $\tau=\cup_{\vec v\in R_{\vec w}}T^{\vec v}F$ of $X$ is called a {\it Rohlin tower of shape
$R_{\vec w}$} with $\vec w\in\mathbb N^d$ if there exists a set
$F\in\mathcal M$ with the property that $T^{\vec v}F\cap T^{\vec
u}F=\emptyset$ for all $\vec v,\vec u\in R_{\vec w}$.  We call the set
$F$ the {\it base of the tower}.

\begin{thm}\label{main}{\em (The $\mathbb Z^d$ Alpern Multi-tower Theorem for
rectangles)} Let $\vec w^1,\vec w^2,\cdots\in\mathbb N^d$, and
$p_1,p_2,\cdots\in\mathbb R^+$ satisfy
\begin{equation}\label{relpr}
\text{for all } i=1,\cdots,d\qquad  \gcd(w^1_i,w^2_i,\cdots)=1
\qquad\text{and}
\end{equation}
\begin{equation}\label{givdistr}
\sum_{j=1}^\infty p_j = 1.
\end{equation}

Then given any free and measure preserving $\mathbb Z^d$ action $T$
of a Lebesgue probability space $(X,\mathcal M,\mu)$ for all $j\ge
1$, there are Rohlin towers $\tau^j$ of shape $R_{\vec w^j}$ with
\begin{gather}
\mu (\tau^j)=p_j, \label{distr}\\
\tau^j\cap\tau^{j'}=\emptyset\text{ if } j\neq j',\label{disjnt}\\
\text{and }  \cup_{j=1}^\infty \tau^{j}=X.\label{noslush}
\end{gather}
\end{thm}

In the case where $d=1$ Theorem \ref{main} is Alpern's original result.
The numbers $w^1,w^2,\cdots$ represent the heights of the
towers and condition (\ref{relpr}) states that their greatest common
divisor is $1$.  This is clearly a necessary assumption since a
non-trivial common divisor of the heights of the towers implies
non-ergodicity of some power of the transformation $T$.  A similar
argument can be made in the higher dimensional case, where a
non-trivial common divisor of corresponding dimensions $\vec e_i$
implies the non-ergodicity of the group element $T^{\vec e_i}$,
where we let $\vec e_1,\cdots,\vec e_d$ denote the standard basis of
$\mathbb Z^d$. It is worth noting that the higher dimensional result
does not require additional constraints on the relationship between
the dimensions of the towers in directions $\vec e_i$ and $\vec e_j$
if $i\neq j$. Namely condition (\ref{relpr}) is a sufficient
condition for all dimensions $d\ge 1$.

The organization of the paper is as follows.  We begin by
considering Theorem~\ref{main} in the case of finitely many rectangles.
This case allows us to highlight the connection to symbolic tiling
spaces, the mixing condition on these systems necessary to prove the
result, and the ease with which we can guarantee the correct
distribution of tiles. Section~\ref{finite} introduces the tiling
spaces and contains the reduction of the proof of Theorem~\ref{main}
to finding an invariant measure on the tiling system for the case of
finitely many rectangles (Theorem~\ref{real}).
Section~\ref{reduction} contains necessary definitions and results
related to mixing properties of shifts of finite type.
Section~\ref{finitephi} contains the proof of Theorem~\ref{real} and
Section ~\ref{phi} contains the proof of Theorem~\ref{main} in the
case of countably infinite rectangles.

The author wishes to thank the Department of Mathematics at George Washington University for their hospitality while most of this work was being completed.  The author also wishes to thank V. S. Prasad for bringing Prikhodko's work to our attention and for fruitful conversations about that work.

\section{Theorem~\ref{main} in the finite case and generalized domino tiling shift spaces}\label{finite}

We begin by defining generalized domino tiles and their associated symbolic
systems.  Fix $k\in\mathbb N$, $k\ge 2$ and vectors $\vec
w^1,\cdots,\vec w^k\in \mathbb N^d$ satisfying the conditions of
Theorem~\ref{main}. Define the vector $\vec W\in\mathbb N^d$ with
$W_i=\prod_{j=1}^k w_i^j$. We define $k+1$ domino tiles:  the
dominos
\begin{equation*}
\tau^j=\prod_{i=1}^d[0,w_i^j-1)
\end{equation*}
for $j=1,\cdots k$ and an additional domino
\begin{equation*}
\tau^P=\prod_{i=1}^d [0,W_i-1)
\end{equation*}
which we will refer to as the {\em large domino}.

Consider all tilings of $\mathbb R^d$ with these dominos subject to
the condition that the vertices of the dominos lie on the integer
lattice.  We refer to such tilings as {\em ($\vec w^1,\cdots, \vec
w^k$) domino tilings}. These tilings can easily be coded into a
one-step $\mathbb Z^d$ shift of finite type $Y(\vec w^1,\cdots,\vec
w^k)$ which we call the $\vec w^1,\cdots,\vec w^k$ {\em domino
tiling shift}. We will be working extensively with this shift and
thus we will give some details of one such coding to establish
necessary notation and terminology.

Define the alphabet of the shift by
$\mathcal A=\{\tau^j_{\vec v}:1\le j\le k, \vec v\in R_{\vec
w^j}\}.$
The set of forbidden blocks of $Y$ are defined first to ensure
that for each $1\le j\le k$ the symbols $\tau^j_{\vec v}$ can be placed
together so that there is a block $\tau^j$ in $\mathcal A^{R_{\vec
w^j}}$ which is a symbolic representation of the domino $\tau^j$. In
the symbolic domino $\tau^j$ the symbol $\tau^j_{\vec 0}$ is taken
to be the origin and appears in the lower left hand vertex of the
block. Then for each $\vec v\in R_{\vec w^j}$, the symbol
$\tau^j_{\vec v}$ appears in position $\vec v$ relative to the
origin. We define the symbolic block corresponding to the domino
$\tau^P$ analogously. A two dimensional example of the alphabet of a
generalized domino tiling shift and the symbolic domino blocks are
shown in Figure \ref{Fi:alphabet}.

\begin{figure}[h!]
\scalebox{0.5}{\includegraphics{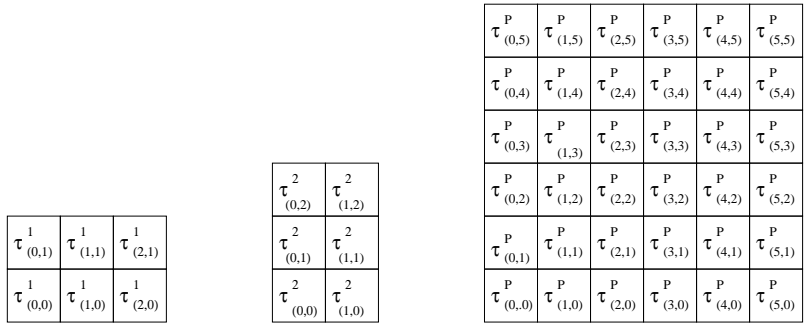}} \caption{The
alphabet and domino tiles of $Y(\vec w^1,\vec w^2)$ with $\vec
w^1=(3,2)$ and $\vec w^2=(2,3)$.} \label{Fi:alphabet}
\end{figure}

In addition, the forbidden blocks are defined so that the words in
$Y$ correspond to a tiling of $\mathbb R^d$ by the dominos
$\tau^j$ and $\tau^P$.
 For example, symbols of the form
$\tau^j_{(w_1^j-1,\cdots)}$ can only be followed horizontally by
symbols of the form $\tau^{j'}_{(0,\cdots)}$, and so forth.

For $y\in Y$ and $\vec v\in\mathbb Z^d$, let $y[\vec v]\in\mathcal
A$ denote the symbol occurring in the point $y$ at position $\vec
v$. We will prove the following result about domino tiling shifts.

\begin{thm}\label{real}
Fix $k\in\mathbb N, k\ge 2$.  Let $\vec w^1,\cdots,\vec w^k$, and $p_1,\cdots,p_k$ satisfy the conditions of
Theorem \ref{main}.  Set $Y=Y(\vec w^1,\cdots,\vec
w^k)$ and let $\mathcal A$ denote the alphabet of $Y$.

Then given any free, measure preserving $\mathbb Z^d$ action $T$
on $(X,\mathcal M,\mu)$ there is a partition
$\phi:X\rightarrow\mathcal A$ so that  for a.e. $x\in X$ there is
a point $y_x\in Y$ with
\begin{equation}\label{equiv}
\phi(T^{\vec v}(x))=y_x[\vec v]
\end{equation}
and for $j=1,\cdots,k$
\begin{equation}\label{msrs}
\mu(\phi^{-1}(\tau^j))<\min_{i=1,\cdots,k} p_i.
\end{equation}
\end{thm}

We obtain the following immediate corollary.

\begin{cor}
Theorem~\ref{main} holds in the case of finitely many towers.
\end{cor}
\begin{proof}
Notice that (\ref{equiv}) guarantees that for all $j=1,\cdots,k$ the
set $\phi^{-1}(\tau^j)$ is a Rohlin tower of shape $R_{\vec w^j}$
and
\begin{equation}\label{disj1}
\phi^{-1}(\tau^j)\cap\phi^{-1}(\tau^{j'})=\emptyset
\end{equation}
when $j\neq j'$. Similarly $\phi^{-1}(\tau^P)$ is a Rohlin tower of
shape $R_{\vec W}$  and
\begin{equation}\label{disj2}
\phi^{-1}(\tau^P)\cap\phi^{-1}(\tau^j)=\emptyset
\end{equation}
for all $j$.

By (\ref{msrs}) each tower $\phi^{-1}(\tau^j)$ takes up less than
its required distribution.  In particular, the numbers
$\alpha_j=(p_j-\mu(\phi^{-1}(\tau^j))$ are all positive. We will
construct additional towers of each shape $R_{\vec w^j}$ with
measure $\alpha_j$ by decomposing the large tower
$\phi^{-1}(\tau^P)$.

Clearly $\mu(\tau^P)=\Sigma \alpha_j$. We partition
$E_P=\phi^{-1}(\tau^P)$,  the base of the large tower, into $k$ measurable subsets, $E^j_P$, each of measure
$\alpha_j$.  Each $E_P^j$ is the base of a tower of shape $R_{\vec
W}$. By definition $\frac {W_i}{w_i^j}\in\mathbb N$ for all
$i=1,\cdots, d$ so for each $j$ we can think of $R_{\vec W}$ as
consisting of a disjoint union rectangles of shape $B_{\vec w^j}$.
In particular, for $0\le j_i\le \frac{W_i}{w_i^j}$, sets of the form
\begin{equation*}
T^{j_1w_1^j\vec e_1+j_2w_2^j\vec e_2+\cdots+j_dw_d^j\vec
e_d}(E_P^j)
\end{equation*}
are all bases of pairwise
disjoint Rohlin towers of shape $R_{\vec w^j}$.  For each $j$ we
define $E_j$ to be the union of all these base sets with
$\phi^{-1}(\tau^j_0)$.

The sets $E_j$ are now bases of Rohlin towers of shape $R_{\vec
w^j}$.  Denoting the resulting tower by $\tau^j$ we now have
$\mu(\tau^j)=p_j$. It is clear from (\ref{disj1}), (\ref{disj2}) and
our construction that (\ref{disjnt}) and (\ref{noslush}) are also
satisfied.
\end{proof}

\section{Good Brick Walls and Uniform Filling Sets of Domino Tiling Shifts}\label{reduction}

The proof of Theorem \ref{real} is an application of the ideas in
\cite{RS4}. There (Corollary 1.3 of Theorem 1.1) it is shown that
given any free, measure preserving, and ergodic $\mathbb Z^d$ action
$T$ on a Lebesgue space $(X,\mathcal M,\mu)$  conclusion
(\ref{equiv}) of Theorem \ref{real} holds for any shift of finite
type $Y$ which has a {\it uniform
  filling set}, defined below. Here we modify that proof, first to
eliminate the need for ergodicity in the argument, and second to also obtain
conclusion (\ref{msrs}) of Theorem \ref{real}.

We begin by defining a uniform filling set for a shift of finite
type. Let $\vert A\vert$ and $A^c$ denote the cardinality and
complement of a set $A$ respectively.  Given a rectangle $R_{\vec
w}\subset\mathbb Z^d$ let $R^\ell_{\vec w}$ denote the rectangle
$R_{\vec w}$ together with an outer filling collar of uniform width
$\ell$.  More formally,
\begin{equation*}
R^\ell_{\vec w}=\{\vec v\in\mathbb Z^d: -\ell\le v_i\le
w_i-1+\ell, 1\le i\le d\},
\end{equation*}

\begin{defn}
A shift of finite type $Y$ has a {\em uniform filling set} if
there is a translation invariant subset $Z$ of $Y$ with $\vert
Z\vert>1$ and an integer $\ell>0$ with the property that given any
$\vec w\in\mathbb N^d$ and $z,z'\in Z$, there is a point $y\in Y$
with the property that
\begin{align*}
y[R_{\vec w}]=z[R_{\vec w}]
\end{align*}
and
\begin{equation*}
y[(R^\ell_{\vec w})^c]=z'[(R^\ell_{\vec w})^c].
\end{equation*}
We call $\ell$ the {\em filling length} of the set $Z$.
\end{defn}

\begin{defn}
Let $\vec w^1,\cdots,\vec w^k\in\mathbb N^d$, with $k\ge 2$ and
let $Y=Y(\vec w^1,\cdots,\vec w^k)$. Let $\tau$ be a domino of shape $R_{\vec w}$.
The {\em good brick wall with domino $\tau$} is the periodic word
$y\in Y$ with the property that for all $\vec v\in\mathbb Z^d$
\begin{equation*}
y[\vec v]=\tau_{(v_1 \mod w_1,\cdots, v_d \mod w_d)}.
\end{equation*}
\end{defn}
In our arguments we will be using the good brick wall with the
special domino $\tau^P$.  Figure \ref{Fi:goodbrick} shows part of
such a good brick wall $y^P$ for a two-dimensional example.

\begin{figure}
\scalebox{0.4}{\includegraphics{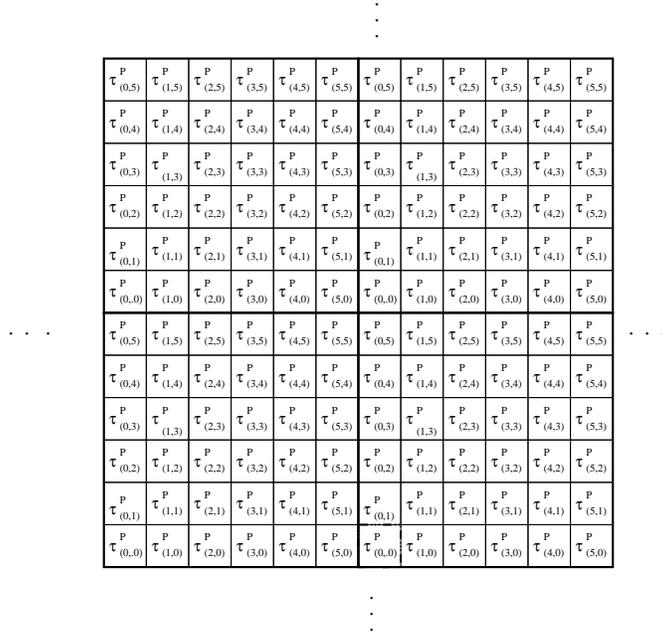}} \caption{The
good brick wall $y^P$ in $Y(\vec w^1,\vec w^2)$ with $\vec
w^1=(3,2)$
  and $\vec w^2=(2,3)$. In this case $\vec W=(6,6)$.  The origin, $y^P[\vec 0]$ is indicated by a symbol
with dashed boundary  labelled $\tau^P_{(0,0)}$. The blocks
corresponding to copies of the domino $\tau^P$ are indicated by bold
boundaries. } \label{Fi:goodbrick}
\end{figure}

\begin{thm}\label{ufs}
Let $Y$ be the $\vec w^1,\cdots,\vec w^k$ domino tiling shift for $\vec w^1,\cdots,\vec w^k\in\mathbb N^d, k\ge 2$ which satisfy
the conditions of Theorem \ref{main}. Let $y^P\in Y$ be the good brick wall with
domino $\tau^P$. Then the set
\begin{equation*}
Z=\cup_{\vec v\in\mathbb Z^d} T^{\vec v}(y^P)
\end{equation*}
is a uniform filling set for $Y$.
\end{thm}

\begin{proof} For notational convenience we give the proof for the
case $d=2$.  The argument for higher dimensions is analogous.
Since the vectors $\vec w^1,\cdots,\vec w^k$ satisfy (\ref{relpr})
of Theorem \ref{main}, there exists $R\in\mathbb N$ with the
property that for all integers $r\ge R$ there exist
non-negative integers $a_j,b_j$ with $j=1,\cdots, k$ such that
\begin{equation}\label{numthry}
r=\sum_{j=1}^k a_jw_1^j=\sum_{j=1}^k b_jw^j_2.
\end{equation}

Let $\ell = R+2W_1+2W_2$, where as before $\vec
W=(W_1,W_2)=(\prod_{j=1}^kw_1^j,\prod_{j=1}^kw_2^j)$. We will show
that $Z$ is a uniform filling set with filling length $\ell$.

Let $z,z'\in Z$, and a rectangle $B_{\vec u}+\vec v\subset\mathbb
Z^2$ be given, with $\vec u\in\mathbb N^2$ and $\vec v\in\mathbb
Z^2$. Let
\begin{equation*}
{\bf a}=z[B_{\vec u}+\vec v]\qquad\text{and}\qquad {\bf
b}=z'[(B^\ell_{\vec u}+\vec v)^c].
\end{equation*}

First notice that ${\bf a}$ might contain incomplete dominos (see
Figure \ref{Fi:badcut}).  These incomplete copies of $\tau^P$ can be
completed by using at most $W_1-1$ columns on either side of ${\bf
a}$ and at most $W_2-1$ rows on the top and bottom of ${\bf a}$.
Similarly, any incomplete dominos that might occur in ${\bf b}$ can
be completed by using at most another $W_1-1$ rows and $W_2-1$
columns of the filling collar.

\begin{figure}
\scalebox{0.4}{\includegraphics{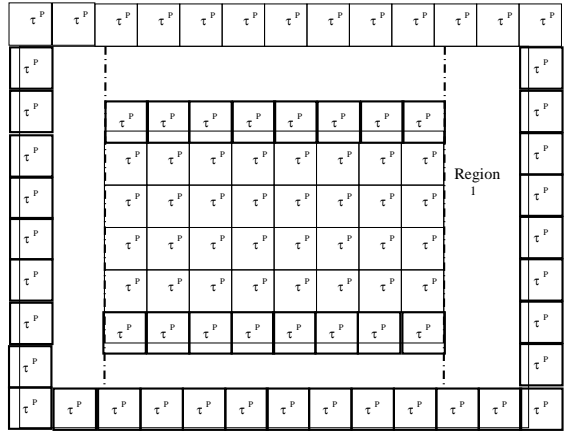}} \caption{The copies
of the domino $\tau^P$ which were incomplete in the original blocks
${\bf a}$ and $\bf b$ are shown in bold. The four rectangular
regions are indicated with a dashed line boundary. The vertical
boundary of Region 1 contains $8$ copies of $\tau^P$, so each
occurrence of the domino $\tau^j$ necessary to tile the width of
Region $1$ will require a column containing $\frac{8*W_2}{w_2^j}$
copies of the domino.} \label{Fi:badcut}
\end{figure}

The resulting new filling collar has width greater than $R$. Notice
also that each piece of its boundary is the union of the boundaries
of complete copies of the domino $\tau^P$. Thus, the filling collar
can be decomposed into four rectangles with the property that each
rectangle either has width greater than $R$ and height a multiple of
$W_2$, or has height greater than $R$ and width a multiple of $W_1$
(see Figure \ref{Fi:badcut}.)

We show how to fill in the tiling for the rectangle labelled Region
$1$ in Figure \ref{Fi:badcut}, the argument for the other regions is
identical. Using (\ref{numthry}) we obtain non-negative integers
$a_j$ so that the width of the rectangle can be written as
$\sum_{j=1}^ka_jw_1^j$. Working from left to right, we start tiling
Region $1$ by first using $a_1$ copies of the tile $\tau^1$ to tile
a strip of width $a_1w_1^1$.  Since $\tau^1$ has height $w_2^1$, and
the height of Region $1$ is a multiple of $W_2$, and therefore of
$w_2^1$, we can tile this strip without gaps or overlap with the
tile $\tau^1$. without overlapping the completion of block $\bf b$
at all. Next, we tile another strip of width $a_2w_1^2$ using
dominos of type $\tau^2$, etc.

This filling process clearly results in a $\vec w^1,\cdots,\vec
w^k$ domino tiling of the plane, and thus a point in $Y$.
\end{proof}

\section{Proof of Theorem \ref{real}}\label{finitephi}
We first establish some notation.  For $s>0$ and $\vec u\in\mathbb
N^d$ we have already in Section \ref{reduction} defined $R^s_{\vec
u}\subset\mathbb Z^d$ to be $R_{\vec u}$ together with an outer
collar of uniform width $s$.
We denote the outer collar itself by:
\begin{equation*}
\partial^s(R_{\vec u})=R^s_{\vec u}\setminus R_{\vec u}
\end{equation*}
and refer to it as the {\it outer $s$-collar} of $R_{\vec u}$.
Still assuming $s>0$, we define the {\it $s$-interior of $R_{\vec
u}$} by
\begin{equation*}
R^{-s}_{\vec u}=\{\vec v\in R_{\vec u}: s\le v_i\le u_i-1-s, i\le
d\}
\end{equation*}
and we call
\begin{equation*}
\partial^{-s}(R_{\vec u})=R_{\vec u}\setminus R^{-s}_{\vec u}
\end{equation*}
the {\it inner $s$-collar of $R_{\vec u}$}.

For notational convenience we will refer to squares of size $n$,
namely $R_{\vec u}$ with $\vec u=(n,\cdots,n)$, as $R_n$.

Our argument will rely heavily on the following gluing property of
shifts of finite type with uniform filling sets. If a shift of
finite type $Y$ has a uniform filling set $Z$, one can glue blocks
from other points $y$ in the shift into points from $Z$, as long
as the block in $y$ has sufficient overlap with a block from $Z$.
The next lemma is a formal statement of this
fact.
\begin{lem}\label{extend}{\em(Lemma 3.5 from \cite{RS4})}.
Suppose $Y$ is a shift of finite type with step size $s$ and
filling set $Z$ with filling length $\ell$. Fix a rectangle
$R_{\vec u}\subset\mathbb Z^d$, with $\vec u\in\mathbb N^d$.
Suppose $y\in Y$ is such that there exists $z\in Z$ with
\begin{equation*}
y\big[\partial^{-s}\big[R_{\vec
u}\big]\big]=z\big[\partial^{-s}R_{\vec u}\big].
\end{equation*}
Then given $z'\in Z$ we can find a point $y^*\in Y$ such that
\begin{equation}
y^*\big[R_{\vec u}\big]=y\big[R_{\vec u} \big] \mbox{\ and\ }
y^*\big[(R^\ell_{\vec u})^c\big]= z'\big[(R^\ell_{\vec u})^c\big].
\end{equation}
\end{lem}
We refer the reader to \cite{RS4} for the proof.

Given a Rohlin tower of shape $R\subset\mathbb Z^d$ with base $F$
and $x\in X$ with the property that $T^{\vec n}x\in F$ for some
$\vec n\in\mathbb Z^d$ we call $T^{\vec n+R}(x)$ an {\it occurrence
of the tower $T^R(F)$}, and we call $y=T^{\vec n}(x)\in F$ the {\it
base point} of this occurrence.  The set $T^{R}(x)$ will be called
the {\it slice of the tower based at $x$}.

\subsection{The parameters of the construction}
We will construct the function $\phi$ of Theorem \ref{real} as a
limit of a sequence of functions $\phi_i$ defined on levels of
Rohlin towers of increasing measures.   We let $Z$ be as in
Theorem~\ref{ufs}.  We denote the filling length of $Z$ by $\ell$,
and fix $z\in Z$.  Let $\epsilon_i$ be a sequence of positive
numbers with the property that for all $i$
\begin{equation}\label{epsilon}
\epsilon_i<\frac1{4^i}.
\end{equation}
Let $n_i$ be a sequence of positive integers increasing to
infinity with the property that
\begin{equation}\label{ncond}
\frac{2d(\ell+2+n_{i-1})}{n_i}<\frac1{4^i}
\end{equation}
and
\begin{equation}\label{2ncond}
2d\ell\sum_{i=1}^\infty\frac{1}{n_i}< \min_{i=1,\cdots,k}p_i.
\end{equation}
Finally we apply the $\mathbb Z^d$ Rohlin Lemma to obtain
sets $F_i\in\mathcal M$ which are bases of a Rohlin tower of shape
$R_{n_i}$ and error sets $B_i$ with
\begin{equation}\label{error}
\mu B_i<\epsilon_i.
\end{equation}

\subsection{Constructing the maps $\phi_i$}
The function $\phi_1$ will be defined on $R_{n_1}^{-(\ell+1)}$.  For
$\vec v\in R_{n_1}^{-(\ell+1)}$ we set
$\phi_1(T^{\vec v}x)=z[\vec v]$.

For ease of notation we describe the construction of $\phi_2$.
This step encompasses all the details of the general inductive
step.   We will define $\phi_2$ on $T^{R_{n_2}^{-(\ell+1)}}(F_2)$.
For each $x\in F_2$ we define the set
\begin{equation*}
B(x)=\{\vec b_1\in R_{n_2}^{-(\ell+2+n_1)}: T^{\vec b_1}x\in F_1\}.
\end{equation*}

Note that these are base points of occurrences of $T^{R_{n_1}}(F_1)$ in
$T^{R_{n_2}}(x)$
\begin{equation}\label{empty}
\text{which are completely contained in
$T^{R_{n_2}^{-(\ell+2)}}(x)$}.
\end{equation}
We refer to these as {\it good occurrences} of the first stage tower.

Partition $F_2$ into subsets $F_2^1,\cdots,F_2^{k_2}$ such that the set of
indices $B(x)$ is constant on each subset.  For each $t=1,\cdots,k_2$ we refer
to the relevant constant set of indices by $B(t)$.

Fix such a $t$.  For all $x\in F_2^t$ we set
$\phi_2(T^{\vec v}x)=z[\vec v]$
for the following choices of $\vec v$:
\begin{align}
\vec v&\in \bigg(\cup_{\vec b\in B(t)} \vec b+R_{n_1}\bigg)^c\cap T^{R_{n_2}^{-(\ell+1)}}(F_2^t) \label{ambient}\\
\vec v&\in \partial^{-1}[\vec b+R_{n_1}]\text{\ where\ } \vec b\in B(t)\label{safeamb}
\end{align}

Locations specified in (\ref{ambient}) puts symbols from
$z[R_{n_2}]$ onto locations in the $(\ell+1)$-interior of $R_{n_2}$
which are not covered by good copies of the first stage tower. We
call this the {\it ambient word}.  Locations specified in
(\ref{safeamb}) extend the ambient picture from $z[R_{n_2}]$ to the
interior $1$-collar around good first stage towers.

For $\vec v\in \vec b+R^{-(\ell+1)}_{n_1}$ where $\vec b\in B(t)$, namely in locations
that lie in the
domain of $\phi_1$, we set
\begin{equation*}
\phi_2(T^{\vec v} x)=\phi_1(T^{\vec v} x).
\end{equation*}
We note that each good occurrence of a first stage tower now sees
\begin{align}\label{2ndstep}
\text{  symbols from }
&z[R_{n_1}] \text{ in positions } \vec b+R_{n_1}^{-(\ell+1)}\text{ and }\\
\text{ symbols from }&z[R_{n_2}] \text{ in locations }\vec
b+\partial^{-1}(R_{n_1}).\notag
\end{align}
where $\vec b\in B(t)$.

Note that $\phi_2$ is as yet undefined on
a
collar of width $\ell$ in all good occurrences of the first stage
tower.  We call this the {\it filling collar} of the tower.  Since the
collars of width one immediately before and after the filling
collar have been assigned symbols from points in $Z$ we can now
interpolate between them and assign symbols to the filling collar
to obtain a block from $Y$. More formally, using (\ref{2ndstep})
and the fact that $Z$ is a uniform filling set we can now find a
point $z_2^t\in Y$ such that for all $\vec b\in B(t)$
\begin{align*}
z_2^t\big[\vec b+R_{n_1}^{-(\ell+1)}\big]&=z\big[R_{n_1}^{-(\ell+1)}\big]\\
z_2^t\big[\vec b+\partial^{-1}(R_{n_1})\big]&=z\big[\vec b+\partial^{-1}(R_{n_1})\big].
\end{align*}

We set
\begin{equation*}
\phi_2(\vec v) = z_2^t[\vec v]
\end{equation*}
for all $\vec v\in \vec b+\partial^\ell(R^{-(\ell+1)}_{n_1})$.

For the inductive step we need to make two observations.  Since
$z_2^t\in Y$, $\phi_2(T^{R_{n_2}^{-(\ell+1)}}(F_2^t))$ is a block
which occurs in $Y$. In addition, (\ref{empty}) and
(\ref{ambient}) guarantee that
\begin{equation}\label{induct}
\text{locations } \vec v\in\partial^1(R_{n_2}^{-(\ell+2)})\text{ are assigned
  the symbol } z[\vec v].
\end{equation}
Thus, even though $\phi_2(T^{R_{n_2}^{-(\ell+1)}}(F_2^t))$ is not
a block from $Z$, we will be able to glue it into blocks from $Z$
in the next stage of the construction by using Lemma~\ref{extend}.

Repeating the process for all $t=1,\cdots,k_2$ we define $\phi_2$
on $T^{R_{n_2}^{-(\ell+1)}}(F_2)$.

\subsection{Constructing $\phi$} We first identify those points $x\in
X$ for which $\lim_i\phi_i(x)$ does not exist.  These are exactly
those points who lie in the error sets $B_i$ for infinitely many $i$
or lie in the inner $(\ell+1+n_{i-1})$-collar of the $i$th tower for
infinitely many $i$. For each $i$
\begin{equation*}
\mu\bigg(B_i\cup
T^{\partial^{-(\ell+1+n_{i-1})}(R_{n_i}})(F_i)\bigg)<\epsilon_i+\frac{2d(\ell+2+n_{i-1})}{n_i}
\end{equation*}
so by (\ref{epsilon}) and (\ref{ncond}) we can use the easy
direction of the Borell-Cantelli Lemma to conclude that for a.e.
$x\in X$ the requisite limit exists.

To obtain (\ref{equiv}) we argue as follows.  This argument is identical to the convergence argument
given in \cite{Ru1}, we include it here for completeness' sake. Let $G_0$ be the
subset of $X$ consisting of points $x$ for which
$\lim_{i\rightarrow\infty}\phi_i(x)$ exists and who land in the
middle ninth of a tower infinitely often.  Set $G_1=\cup_{\vec
v\in\mathbb Z^2}T^{\vec v}G_0$. $G_1$ is clearly a measurable and
invariant set and $\mu(G_1)=\alpha>\frac1{36}$.

Further, for $x\in G_1$since the $n_i$ grow to infinity we have that for any
$\vec v\in\mathbb Z^d$, there exist infinitely many $i$ such that
$x$ and $T^{\vec v}x$ are in the same slice of the $i$th stage
tower.  For such $i$ clearly
\begin{equation*}
\phi_i(T^{\vec v}x)=S^{\vec v}(\phi_i(x))
\end{equation*}
so $\lim_{i\rightarrow\infty}\phi_i(T^{\vec v}x)$ must exist and
(\ref{equiv}) must hold.

If $\alpha<1$ we can use the same procedure, adjusting the measures $p_1,\cdots,p_k$ as necessary, to obtain a set
$G_2\subset G_1^c$ such that $\mu(G_1\cup G_2)^c<(1-\alpha)^2$.
Countably many such steps yields a countable union of invariant
measurable sets $G_n$ with the property that $\mu(\cup G_n)=1$ and
on each set $\phi$ is defined and satisfies (\ref{equiv}).

To see that (\ref{msrs}) holds, we note that our tiling argument
guarantees that the only part of the space that is not tiled by the
domino $\tau^P$ is the filling collars of the towers. Thus by
(\ref{2ncond}) we have
\begin{equation*}
\mu(\phi^{-1}(\tau^P))>1-\sum_{i=1}^\infty\frac{2d\ell}{n_i}>1-\lim_{i=1,\cdots,k}p_i.
\end{equation*}

\section{Proof of Theorem~\ref{real} in the case of countably many
towers}\label{phi}

Our construction here will give a measurable tiling of the orbits of
the action by finitely many actual size dominos and infinitely many
large dominos.  As before, the actual size small dominos will take
up less than their alloted measure of space and we will use the
large dominos to carve out the required size small towers of
appropriate measures.

There are several new issues we must address.  First, we must ensure
that we have infinitely many large domino towers, since a tower of
size $\tau^j$ can only be constructed out of a large domino whose
dimensions are the product of at least the first $j$ dimensions. Our
tiling construction then has to be modified to incorporate large
dominos of growing scales.

Second, a large domino whose dimensions are a product of the first
$j$ dimensions must be entirely used up when we construct small
towers of size $1,\cdots,j$, since it will be too small to use in
later stages.  We thus have to bound the measures of the larger
dominos appropriately.

Now for the details.  Choose a sequence $n_k\in\mathbb N$ so that
for all $i=1,\cdots,d$,
\begin{align}
gcd(w_i^1,\cdots,w_i^{n_1})&=1,\label{gcdkickedin}\\
\epsilon_k =\sum_{j=n_k+1}^\infty p_j&<\frac 1{8^k},\qquad{and}\label{tailcontrol}\\
\sum_{j=n_k}^{n_{k+1}}p_j&>\sum_{j=n_{k+1}}^\infty p_j.\label{cntgetall}
\end{align}
Following the notation established in Section ~\ref{reduction} we
let $Y_k=Y(\vec w^1,\cdots,\vec w^{n_k})$ and $Z_k$ denote the
uniform filling set of $Y_k$ generated by the large domino
$\tau^{P_k}$.  We let $Y_\infty$ denote the shift on the infinite
alphabet required to define all tiles $\tau^1,\cdots$.  As in the
finite case we will first prove that we can construct a factor map
$\phi:X\rightarrow Y_\infty$ satisfying conditions (\ref{equiv}) and
(\ref{msrs}).

We prove that any block from $Z_k$ can be glued into a word from
$Z_1$ using a collar whose size depends only on $k$ and not on the
size of the block itself.  This is the key ingredient which will
allow us to tile with growing sizes of large dominos.

\begin{cor}\label{genfil}\em{(of Theorem~\ref{ufs})}.  Given $k\in\mathbb N$, there exists $l_k\in\mathbb N$ such that for all $\vec w\in\mathbb N^d$, $z\in Z_k$ and $z'\in Z_1$, there exists a point $y\in Y_k$ such that
\begin{align*}
y[R_{\vec w}] &= z[R_{\vec w}]\\
y[(R^\ell_{\vec w})^c]&=z'[(R^\ell_{\vec w})^c]
\end{align*}
with the property that $y[R_{\vec w}^\ell\setminus R_{\vec w}]$ is a union of tiles only of type $\tau^1,\cdots,\tau^{n_1}$.
\end{cor}
\begin{proof}
The case for $k=1$ is exactly Theorem~\ref{ufs}.  For $k>1$ we note
that the extra width in the collars provides enough space to
complete any necessary tiles of shape $\tau^{P_k}$ in $z[R_{\vec
w}]$, resulting in an inner rectangle whose boundaries have
dimensions of completed rectangles from $Z_k$, and therefore from
$Z_1$.  The remaining filling collar is now large enough so that
together with (\ref{kickedin}) the filling algorithm from the proof
of Theorem~\ref{ufs} can be used verbatim to fill the collar using
only tiles of type $\tau^1,\cdots,\tau^{n_1}$.
\end{proof}

There is also an analogous extension of Lemma~\ref{extend} which we
omit stating.

Choose an increasing sequence of integers satisfying:
\begin{align}
\frac{2d(\ell_k+2+n_{k-1})}{n_k}&<\frac1{4^k}\label{cntcollar}\\
\text{and }
\sum_{k=1}^\infty\frac{2d\ell_k}{n_k}&<\min_{j=1,\cdots,n_1}p_j\label{cntmsr}
\end{align}

We fix $z_i\in Z_i$, and choose $R_i,F_i$ and $B_i$ as before.  For
$i>1$ we divide $F_i$ into two subsets $F_i^{main}$ and $F_i^{tail}$
with the property that
\begin{equation}\label{tailmsr}
\mu(T^{R_{n_i}}(F_i^{\text{tail}}))=\sum_{k=n_i+1}^{n_{i+1}-1}p_k.
\end{equation}

\subsection{Constructing $\phi$}  We begin by constructing $\phi_1$ exactly as before.
We define $\phi_2$ as follows. On $F_2^{\text{main}}$ we proceed
exactly as in the finite case, using $z_1$ to obtain the ambient
word.  For $x\in F_2^{\text{tail}}$ we fix $z_2\in Z_2$ and we set
\begin{equation*}
\phi_2(T^{\vec v}x)=z_2[\vec v]
\end{equation*}
for $\vec v\in R_{n_2}^{-\ell_2+1}$.

To see all the necessary ingredients for the general induction step
we also need to describe the construction of $\phi_3$. Fix $z_3\in
Z_3$ and for $x\in F_3^{\text{tail}}$ we assign
\begin{equation*}
\phi_3(T^{\vec v}x)=z_3[\vec v]
\end{equation*}
for $\vec v\in R_{n_3}^{-\ell_3+1}$. For $x\in F_3^{\text{main}}$ we
proceed as in the finite case using $z_1$ to provide an ambient word
from $Z_1$, noting that any slices of previous stage towers that
appear in $T^{R_3}(x)$ come with sufficiently large filling collars
that can be filled either by Theorem~\ref{ufs} or by
Corollary~\ref{genfil}.  We note that the resulting shapes can be
interpolated into a word from $Z_1$ using the modified version of
Lemma~\ref{extend}.

To define $\phi:X\rightarrow Y_\infty$
we again begin by identifying points $x\in X$ for which $\lim \phi_i(x)$ does not exist.  In addition to points that lie in the error sets $B_i$ infinitely often or in the filling collars of infinitely many towers we must now also include points that land in $F_i^{\text{tail}}$ for infinitely many $i$.  For each $i$ this is a set of measure less than
\begin{equation*}
\epsilon_i+\frac{2d(\ell_i+2+n_{i-1})}{n_i}+\sum_{k=n_i}^{n_{i+1}-1}p_k,
\end{equation*}
which by (\ref{tailcontrol}) and (\ref{cntcollar}) is less than
$\frac{3}{8^i}$.  We can thus proceed to define the map $\phi$
exactly as in the finite case.

At the end of this construction we obtain a decomposition of $X$ into towers of shapes $\tau^1,\cdots,\tau^{n_1},\tau^{P_1},\tau^{P_2},\cdots$.
 As before, by (\ref{cntmsr}) none of the first set of small towers take up more than their allotted part of space.
In addition, we have by our choice of $\epsilon_i$,
(\ref{tailcontrol}), (\ref{cntgetall}), and (\ref{tailmsr}) that
$0<\mu(\tau^{P_i})<\sum_{j=n_i}^\infty p_j$.  So if we supplement
the measures of the towers $\tau^i$, or create such towers in the
case that $i>n_1$, in increasing order of index we are guaranteed
that we will decompose all of $\tau^{P_i}$ into required shape
towers for all $i\ge 1$, ending with only towers of the desired
shapes and measures.


\def\cprime{$'$}
\providecommand{\bysame}{\leavevmode\hbox to3em{\hrulefill}\thinspace}
\providecommand{\MR}{\relax\ifhmode\unskip\space\fi MR }
\providecommand{\MRhref}[2]{%
  \href{http://www.ams.org/mathscinet-getitem?mr=#1}{#2}
}
\providecommand{\href}[2]{#2}

\end{document}